\documentclass[12pt,a4paper, twoside]{amsart}
\usepackage{latexsym}
\usepackage[english]{babel}
\usepackage[all]{xy}
\usepackage{graphicx}
\usepackage{dsfont}
\usepackage{german} 
\usepackage[latin1]{inputenc} 
\usepackage{amsfonts, verbatim}
\usepackage{amsmath}
\usepackage{amsthm}
\usepackage{amssymb}
\newcommand{\bn}{\bigskip\noindent}

\newtheorem{theorem}{\bf Theorem}[section]
\newtheorem{lemma}[theorem]{\bf Lemma}

\newtheorem{proposition}[theorem]{\bf Proposition}
\newtheorem{remark}[theorem]{\bf Remark}

\newcommand{\G}{\mathcal{G}}
\newcommand{\A}{\mathcal{A}}
\newcommand{\Z}{\mathcal{Z}}
\input amssym.def

\begin{document}
\title{On the work of Ritter and Weiss in comparison with Kakde's approach}

 \bn
\author{Otmar Venjakob}

\address{Universit\"{a}t Heidelberg,  Mathematisches Institut,  Im Neuenheimer Feld 288,  69120
Heidelberg,  Germany,
 http://www.mathi.uni-heidelberg.de/$\,\tilde{}\,$venjakob/}
\email{venjakob@mathi.uni-heidelberg.de}

\maketitle

\section{Introduction}

Almost simultaneously Ritter and Weiss \cite{RWIX} on the one hand and Kakde \cite{kakde} on the
other hand gave a proof of the non-commutative  Iwasawa main conjecture over totally real fields
for the Tate motive under the assumption that a certain $\mu$-invariant vanishes as has been
conjectured also by Iwasawa. Actually the article of Ritter and Weiss appeared slightly earlier on
the arXive and deals with the case of one-dimensional $p$-adic Lie-extensions, which can be
combined with Burns' well-known insight in \cite[thm.\ 2.1]{burns-MC} (this crucial idea has been
presented by Burns already during a seminar at the University of Kyoto in early 2006)  based on
Fukaya and Kato's result \cite[prop.\ 1.5.1]{fukaya-kato} to obtain the conjecture for general
$p$-adic (admissible) Lie-extensions.    Kakde's paper, which got its final version practically
during this workshop, contains - using the same result of Fukaya and Kato - directly a full proof
in the general situation. Various special cases have been known by the work of (in alphabetical
order) Hara \cite{hara}, Kakde \cite{kakde-cases}, Kato \cite{kato-heisenberg} as well as Ritter
and Weiss \cite{RWVIII}.

  Since in both
approaches the reduction steps from the general case of one-dimensional extensions to the
$l$-elementary extensions (in the language of Ritter and Weiss) or to essential pro-$p$-extensions
(in the language of Kakde) are based on the same principles (even though at different places: for
Ritter and Weiss with respect to the Hom-description while in Kakde's case for the $K$-theory) and
since again the generalisation from the pure pro-$p$ case to this slightly more general case
follows certain standard techniques, see \cite{suja-workshopms}, we restrict in this survey from
the very beginning to the one-dimensional pro-$p$-case. In these notes we shall use the same
notation as in \cite{sch-venkakde}, but for the convenience of the reader we have collected the
crucial notation from \cite{RWIX} in a glossary below comparing it with Kakde's and our notation,
respectively.

Since during the workshop and hence in this volume Kakde's approach has been discussed in great
detail and we may hence  assume greater familiarity with his  techniques, many comments in this
comparison will be made from the perspective of Kade's point of view. So, we want to stress that
this certainly does not reflect the historical development as Kakde has probably been influenced by
a couple of ideas from Ritter and Weiss. E.g.\ the analytic techniques for the actual proof that
the abelian $p$-adic $L$-functions satisfy the required conditions in order to induce the
``non-commutative'' $p$-adic $L$-function has been applied first by Ritter and Weiss as well as by
Kato (in an unpublished preprint about Heisenberg type $p$-adic Lie group extensions). For that
reason we won't discuss this analytic part at all in this comparison because again the methods are
essentially the same, see \cite[\S 3]{RWIX} and \cite[\S 6]{kakde}, respectively.

One evident difference  between the two works consists of the way of presentation: While Kakde, who
also had partial results in previous publications,  delivers an almost self-contained account of a
full proof of the main conjecture (even dealing with arbitrary admissible $p$-adic Lie extensions)
in one    ingenious paper, the impressive work of Ritter and Weiss is spread over at least 10
articles \cite{RWt,RWI, RWII, RWIII, RWIV, RWV, RWVI, RWVII, RWVIII, RWIX}, viz naturally in the
way as their theory has been developed over the recent decade. The last article, which contains the
general result, is rather an instruction how to modify and extend the proofs of earlier results (in
less general cases) in previous publications combined with an extensive discussion of the new
M\"{o}bius-Wall-congruence in order to complete the proof in the general case than a self-contained
proof. A general outline of the overall strategy of the proof is missing and the reader is forced
even to collect the notation from all the other articles. This is somewhat unfortunate as otherwise
the strategy of their nice proof could have been much more easily accessible for the
reader.\footnote{In order to remedy this Ritter and Weiss are apparently thinking about writing a
lecture note volume about their approach.}

Another obvious difference is the fact that while Kakde describes $K_1(\Lambda(\mathcal{G}))$
modulo $SK_1(\Lambda(\mathcal{G}))$ in terms of certain relations (some call them congruences)
among elements in the units of Iwasawa algebras for certain abelian subquotients of $\mathcal{G}$
Ritter and Weiss are only interested whether the particular system of abelian pseudomeassures is in
the image of $K_1. $ But in the last section of this survey we try to clarify what can be proved by
the methods of Ritter and Weiss towards describing the image of $K_1$ in the product of the $K_1$'s
of appropriate  subquotients of $\mathcal{G}$ in the manner of Kakde.

I am very grateful to Cornelius Greither and Mahesh Kakde for reading an earlier version and
suggesting a couple of improvements.

\section{Glossary}  \hspace{.5cm}

In the following the notation in the different columns usually denote the same object, but
sometimes it only indicates that they are just closely related.

  $$
  \begin{array}{cp{2cm}cc}
   \mathrm{Kakde}                                        && \mathrm{Ritter/Weiss}\\
   \hline &&\\
    &&\\ 
    p \neq 2                                    && l \neq 2\\
    &&\\
    F_\infty/F                                  && K/k\\
    &&\\
    \mathcal{G} = G(F_{\infty} / F)            && G=G(K/k)\\
    &&\\
    F^{\mathrm{cyc}}                            && k_\infty\\
    &&\\
    \Gamma                                           && \Gamma_k\\
    &&\\
    H=G(F_\infty/F^{\rm cyc})                   && H\\
    &&\\
    A:=A(\mathcal{G}):=\Lambda(\mathcal{G})_S=\Lambda(\mathcal{G})_{(p)}       && \Lambda_{\bullet} (G)\\
    &&\\
    B:=B(\mathcal{G})=\widehat{\Lambda(\mathcal{G})_S} && \Lambda_{\wedge} (G)\\
    &&\\
    \Lambda({\mathcal{G}})_{S^*}=A(\mathcal{G})[\frac{1}{p}] && \mathcal{Q}G\\
    &&\\
    \mathcal{O}/\mathbb{Z}_p\ \mathrm{unramified,}\; (\Lambda=\Lambda_\mathcal{O})&& \frak{O}\\
        &&\\
    \Sigma\supseteq\Sigma_{\rm ram}(F_\infty/F)
    \;\;\;\;\; {\mbox{finite set of}\atop{\mbox{primes of $F$}} }
    && S\supseteq S_{\rm ram}(K/k)\cup S_\infty
    \;\;\;\;\;  {\mbox{finite set of}\atop{\mbox{primes of $k$}}}\\
    &&\\
    K_0(\Lambda(\mathcal{G}),\Lambda(\mathcal{G})_S)=
    K_0(\mathcal{H}_S)=K_0(\mathcal{C}_S) && K_0T\,\Lambda(\mathcal{G})\\
    &&\\
    {[R{\Gamma}(G(F_\Sigma/F_\infty),\ \mathbb{Q}_p/\mathbb{Z}_p)^{\vee}]}     && \mho=\mho_S\mbox{ (``mhO'' or ``agemO'')}\\
    &&\\
    \zeta_{F_\infty/F}     && \Theta,\lambda_{F_\infty/F}\ (\mathrm{pseudomeasure)}  \\
                           && \mbox{if } F_\infty/F \mbox{ is abelian}\\
    &&\\
    \mathcal{O}[[ Cong\ (\mathcal{G})]]=\Lambda_\mathcal{O}  (\mathcal{G}   )^{\mathrm{ab}}
                           && T \Lambda_{\frak{O}}(\mathcal{G})
  \end{array}
  $$
 \goodbreak
  $$
  \begin{array}{ccc}
       \parbox[t]{8cm}{$\mu=0 \Leftrightarrow $ There exists an open pro-$p$ subgroup $ H'\subset H$
         such that $X_\mathcal{F}=G(\mathfrak{F}_\Sigma/\mathfrak{F})^{\mathrm{ab}}(p)$   is finitely generated
       over $ \mathbb{Z}_p$   for $\mathfrak{F}= F_\infty^{H'}$ and $\mathfrak{F}_\Sigma$ the maximal outside
 $ \Sigma$   unramified extension of $\mathfrak{F}.$}

     &  & \parbox[t]{6cm}{$\phantom{mmmm}\mu_{{\Gamma}'}(X_{F_\infty})=0\phantom{mmmmm}$
          for some open subgroup $ \Gamma'\subset\Gamma$}
 \end{array}
  $$

  \[\begin{array}{ccc}
      \mbox{Kakde - Schneider/Venjakob or this article }& & \mbox{Ritter/Weiss}\\
       \hline
       &&\\
     \pi^U_V,\; pr^U_V  & U\twoheadrightarrow V& \mathrm{defl}^V_U\\
     &&\\
     N^U_V & V\subseteq U & \mathrm{res}^V_U\\
     &&\\
    \tilde{\mathrm{Tr}}^U_V && \mathrm{Res}^V_U \\
    &&\\
     L,\; L_B && \mathbb{L}\\
     &&\\
     \mathcal{L} && {\bf L},\; \prod_U \mathbb{L}_U \\
\end{array}
\]
\section{2-extensions  versus perfect complexes}

The statement of the Main Conjecture (MC) affirms in both approaches the existence of an element
$\Theta,$ respectively $\zeta \in K_1(\Lambda(\mathcal{G})_S)$ satisfying firstly a certain
interpolation property we shall discuss later and secondly is mapped under the connecting
homomorphism of the localisation sequence of $K$-theory
\[K_1(\Lambda(\mathcal{G})_S)\to K_0(\mathcal{H}_S)\cong K_0(\mathcal{C}_S)\]
  to the class $\mho_S,$ respectively $[R{\Gamma}(G(F_\Sigma/F_\infty),\
\mathbb{Q}_p/\mathbb{Z}_p)^{\vee}  ].$ To be more precise, one uses in general the localisation
with respect to the Ore set $S^*$ for the statement of the MC, but assuming the vanishing of $\mu,$
it is easily reduced in both approaches to the above statement.
Thus the first question which naturally arises is:\\

Why do the classes of $R{\Gamma}(G(F_\Sigma/F_\infty),\ \mathbb{Q}_p/\mathbb{Z}_p)^{\vee} $ and
$\mho=\mho_\Sigma$ in  the relative $K_0$ agree?

Let $\Omega$ denote the maximal outside $\Sigma$ unramified $p$-extension of $F_\infty^\Gamma;$
here we fixed a splitting $\Gamma\subseteq \mathcal{G}$ of the projection $G\twoheadrightarrow
\Gamma.$ We denote the Galois group $G(\Omega/F)$ by $G'_\Sigma$ and note that this group is
well-known to be (topologically) finitely generated. Hence, choosing $d$ generators we obtain the
following commutative diagram

\[\xymatrix{
    & 1 \ar[d]_{ }   & 1 \ar[d]_{ }   &  &   \\
  & N \ar[d]_{ } \ar@{=}[r]   & N \ar[d]_{ }   &   &   \\
  1   \ar[r]^{ } & R \ar[d]_{ } \ar[r]^{ } & F_d\ar[d]_{ } \ar[r]^{ } & {\mathcal{G}} \ar@{=}[d]_{ } \ar[r]^{ } & 1 \phantom{,} \\
  1  \ar[r]^{ } & H'_\Sigma\ar[d]_{ } \ar[r]^{ } & G'_{\Sigma} \ar[d]_{ } \ar[r]^{ } & {\mathcal{G}}    \ar[r]^{ } &1,  \\
    & 1   & 1   &   &    }\]
    in which $F_d$ denotes the free pro-finite group on $d$ elements and all the other groups are
    defined by exactness of the diagram. By \cite[prop.\ 5.6.6]{nsw} there is a canonical complex
    associated to the above diagram
  \[\xymatrix{
    0 \ar[r]^{ } & N^{\mathrm{ab}}(p) \ar[r]^{ } & {\mathbb{Z}_p[[G'_{\Sigma}]]^d} \ar[r]^{ } &   {\mathbb{Z}_p[[G'_{\Sigma}]]}\ar[r]^{ } & 0   }\]
which forms a resolution of $\mathbb{Z}_p$ by projective $\mathbb{Z}_p[[G'_{\Sigma}]]$-modules,
whence flat $\mathbb{Z}_p[[H'_\Sigma]]$-modules. Therefore, the complex
\[\xymatrix{
    0 \ar[r]^{ } & N^{\mathrm{ab}}(p)_{H'_\Sigma} \ar[r]^{ } & {\mathbb{Z}_p[[\mathcal{G}]]^d} \ar[r]^{ } &   {\mathbb{Z}_p[[\mathcal{G}]]}\ar[r]^{ } & 0 ,  }\]
which arises by taking $H'_\Sigma$-coinvariants, represents
\[R{\Gamma}(G(F_\Sigma/F_\infty),\ \mathbb{Q}_p/\mathbb{Z}_p)^{\vee}\cong R{\Gamma}(H'_\Sigma,\
\mathbb{Q}_p/\mathbb{Z}_p)^{\vee}\] by Pontryagin duality and using a theorem of Neumann.

By \cite[prop.\ 5.6.7]{nsw} and using the weak Leopoldt statement $H_2(H'_\Sigma,\mathbb{Z}_p)=0$
there is a quasi-isomorphism
\[\xymatrix{
    0 \ar[r]^{ } & N^{\mathrm{ab}}(p)_{H'_\Sigma} \ar[r]^{ }\ar[d] & {\mathbb{Z}_p[[\mathcal{G}]]^d} \ar[r]^{ } \ar[d]&   {\mathbb{Z}_p[[\mathcal{G}]]}\ar[r]^{ }\ar@{=}[d] & 0 \\
     0 \ar[r]^{ }  & 0\ar[r] & Y \ar[r]^{d_Y} & {\mathbb{Z}_p[[\mathcal{G}]]} \ar[r] & 0,  }\]
where $Y:= I_{G'_{\Sigma}}/I_{H'_\Sigma} I_{G'_{\Sigma}}$ and the map
$d_Y:Y\to\mathbb{Z}_p[[\mathcal{G}]]$ factorises over the augmentation ideal as
\[Y\twoheadrightarrow I_\mathcal{G}\hookrightarrow \mathbb{Z}_p[[\mathcal{G}]].\]

Now consider the exact sequence of (vertical) complexes

\[\xymatrix{
  0   \ar[r]^{ } &   {\mathbb{Z}_p[[\mathcal{G}]]}\ar@{=}[d]_{ } \ar[r]^{\psi } & Y \ar[d]_{ d_Y} \ar[r]^{ } & M \ar[d]_{ } \ar[r]^{ } & 0  \\
  0 \ar[r]^{ } & {\mathbb{Z}_p[[\mathcal{G}]]} \ar[r]^{\psi' } & {\mathbb{Z}_p[[\mathcal{G}]]} \ar[r]^{ } & N \ar[r]^{ } & 0,   }\]
  where $\psi$ is a suitably chosen homomorphism such that   $\psi':=d_Y\circ\psi,$ which has by definition image in
  $I_\mathcal{G},$ is injective. Then,  by definition
  \[\mho=[M]-[N] = [M\to N]\in K_0(\mathcal{H}_S)\cong K_0(\mathrm{Ch}^b(\mathcal{H}_S)),\]

  where $\mathcal{H}_S$ denotes the exact category of $S$-torsion $\Lambda(\mathcal{G})$-modules with
  finite projective dimension while $\mathrm{Ch}^b(\mathcal{H}_S)$ denotes the category of bounded
  complexes in $\mathcal{H}_S$ and the above identification is shown in \cite[II.~ thm.\ 9.2.2]{weibel-k}. We also shall write $\mathcal{C}_S:=\mathrm{Ch}_{\mathrm{perf},S}(\Lambda(\mathcal{G}))$ for the category of perfect
   complexes of $\Lambda(\mathcal{G})$-modules which become acyclic after tensoring  with
   $\Lambda(\mathcal{G})_S.$ Under the identification
   $K_0(\mathrm{Ch}^b(\mathcal{H}_S))=K_0(\mathrm{Ch}_{\mathrm{perf},S}(\Lambda(\mathcal{G})))$ (cp.\ \cite[II exerc.\ 9.2, V exerc.\ 3.14]{weibel-k}) we then have
\begin{eqnarray*}
[M\to N]&=&[Y\to \mathbb{Z}_p[[\mathcal{G}]]\;] -[\mathbb{Z}_p[[\mathcal{G}]]\to \mathbb{Z}_p[[\mathcal{G}]]\;] \\
&=&[Y\to \mathbb{Z}_p[[\mathcal{G}]]\;]   \\
&=&[R{\Gamma}(G(F_\Sigma/F_\infty),\ \mathbb{Q}_p/\mathbb{Z}_p)^{\vee}]
\end{eqnarray*}
because the complex $\mathbb{Z}_p[[\mathcal{G}]]\to \mathbb{Z}_p[[\mathcal{G}]]$ is acyclic and
quasi-isomorphisms induce identities in $K_0.$ Thus we have shown the following

\begin{proposition}
\[\mho=[R{\Gamma}(G(F_\Sigma/F_\infty),\ \mathbb{Q}_p/\mathbb{Z}_p)^{\vee}]\phantom{m}\mbox{ in }\phantom{m} K_0(\mathcal{H}_S)\cong K_0(\mathrm{Ch}_{\mathrm{perf},S}(\Lambda(\mathcal{G}))).\]
\end{proposition}

  \section{Hom-description}\label{homdescription}
We shall write $ {\rm{Irr}}(\mathcal{G})\ $ for the set of $\ \overline{\mathbb{Q}}_p$-valued
irreducible representations of $\mathcal{G}$     with finite image. Considering elements of $
K_1(\Lambda(\mathcal{G}))$ as maps on (irreducible) representations we obtain a natural
homomorphism
\[\mathrm{Det}: K_1(\Lambda(\mathcal{G}))\to \mathrm{Maps}({\mathrm{Irr}}(\mathcal{G}),\overline{\mathbb{Z}_p}^\times)\]
into maps from ${\mathrm{Irr}}(\mathcal{G})$ to the units of the ring of integers
$\overline{\mathbb{Z}_p}$ of $\ \overline{\mathbb{Q}}_p,$ which allows for example requiring
interpolation properties for elements of the $K$-group. Of course, the target is ``much bigger''
than the image and one crucial question for the proof of the main conjecture consists of finding
the most appropriate  intermediate target in which to work. Ritter and Weiss follow the philosophy
of Fr\"{o}hlich's so called Hom-description:
\[\mathrm{Hom}_{G_{\mathbb{Q}_p}}(R(\mathcal{G}),\overline{\mathbb{Z}_p}^\times)\subseteq  \mathrm{Hom}(R(\mathcal{G}),\overline{\mathbb{Z}_p}^\times)= \mathrm{Maps}({\mathrm{Irr}}(\mathcal{G}),\overline{\mathbb{Z}_p}^\times),\]
$R(\mathcal{G})$ being the group of virtual characters, i.e., the free
   abelian group generated by ${\mathrm{Irr}}(\mathcal{G}),$ and they refine it slightly by an Iwasawa-theoretic variant
\[\xymatrix{
  K_1(\Lambda(\mathcal{G}) \ar[r]^(0.3){\mathrm{Det}} & {\mathrm{Hom}_{G_{\mathbb{Q}_p}, \, R(\Gamma)}} (R(\mathcal{G}), \Lambda_{\overline{\mathbb{Z}_p}}{(\Gamma)}^{\times})}\subseteq \mathrm{Hom}_{G_{\mathbb{Q}_p}}(R(\mathcal{G}),\overline{\mathbb{Z}_p}^\times)\]
whose definition will be recalled in the case of the localised Iwasawa algebra
$\Lambda(\mathcal{G})_{S^*}$ at the end of this section.

In contrast, Kakde uses a version which is closer to the side of $K$-theory
\[\xymatrix{
  K_1(\Lambda(\mathcal{G})) \ar[r]^(0.2){\theta} & {\left(\prod_{U\in S(\mathcal{G},\mathcal{Z})} K_1(\Lambda(U^{\mathrm{ab}}))\right)^\mathcal{G}= \left(\prod_{U\in S(\mathcal{G},\mathcal{Z})}  \Lambda(U^{\mathrm{ab}})^\times \right)^\mathcal{G}}} \]
where $\mathcal{G}$ acts by inner conjugation on $S(\mathcal{G},\mathcal{Z}), $   the  set of all
subgroups $\Z \subseteq U \subseteq \G$ for a (fixed) central subgroup $\Z \subseteq \G$, and among
the Iwasawa algebras: $(x_U)_U\mapsto (g^{-1}x_{gUg^{-1}}g)_U.$ Recall from \cite[\S
4]{sch-venkakde} that the component $\theta_U$ of $\theta$ is the composite
\begin{equation*}
    \theta_U : K_1(\Lambda(\G)) \xrightarrow{\; N^\G_U \;} K_1(\Lambda(U)) \longrightarrow K_1(\Lambda(U^{\mathrm{\mathrm{ab}}})) = \Lambda(U^{\mathrm{\mathrm{ab}}})^\times
\end{equation*}
of the norm map and the homomorphism  induced by the canonical surjection $U \twoheadrightarrow
U^{\mathrm{ab}}.$

By explicit Brauer induction \cite{boltje,snaith} we obtain a splitting $a_\mathcal{G}$ of the
canonical map
\[(\oplus_{U\in S(\mathcal{G},\mathcal{Z})} R(U^{\mathrm{ab}}/\mathcal{Z}))_\mathcal{G} \to R(\mathcal{G}/\mathcal{Z}),\]
which  is induced by sending a tuple $(n_U\chi_U)$ of one-dimensional characters $\chi_U$ of
$U/\mathcal{Z}$ and integers $n_U$ to $\sum n_U
\mathrm{Ind}^{\mathcal{G}/\mathcal{Z}}_{U/\mathcal{Z}}(\chi_U);$ by abuse of notation we shall also
write $\chi_U$ for the associated  character $\mathrm{infl}_{U/\mathcal{Z}}^U(\chi)$ of $U$ by
inflation. Again, $\mathcal{G}$ is acting by inner conjugation: $g(n_U\chi_U)_{U\in
S(\mathcal{G},\mathcal{Z})}:=(n_{gUg^{-1}}\chi_{gUg^{-1}}(g\cdot g^{-1}))_{U\in
S(\mathcal{G},\mathcal{Z})},$ where $\cdot$ denotes the missing argument.   Hence, using also the
fact \cite[lem.\ 92]{kakde} that we have a surjection
\begin{equation}\label{surj}\mathrm{Irr}(\Gamma)\times \mathrm{Irr}(\mathcal{G}/\mathcal{Z})
\twoheadrightarrow \mathrm{Irr}(\mathcal{G}),\quad (\chi,\rho)\mapsto
\mathrm{infl}_\Gamma^\mathcal{G}(\chi)\cdot \mathrm{infl}_{\mathcal{G}/\Z}^\mathcal{G}(\rho)
\end{equation} and hence a surjection
\[R(\Gamma)\otimes_{\mathbb{Z}}R(\mathcal{G}/\mathcal{Z}) \twoheadrightarrow R(\mathcal{G}),\] we may define a homomorphism
\[\mathrm{BrInd}: \left(\prod_{U\in S(\mathcal{G},\mathcal{Z})}  \Lambda(U^{\mathrm{ab}})^\times \right)^\mathcal{G} \to  {\mathrm{Hom}_{G_{\mathbb{Q}_p}, \, R(\Gamma)}} (R(\mathcal{G}), \Lambda_{\overline{\mathbb{Z}_p}}(\Gamma)^{\times}) \]
such that \[\mathrm{BrInd}((x_U)_U)(\rho):=\prod_U
\mathrm{Det}_{U^{\mathrm{ab}}}(x_U)(\chi_U)^{n_U},\] if $a_\mathcal{G}(\rho)$ is represented by
$(n_U \chi_U)_{U\in S(\mathcal{G},\mathcal{Z})}$ (and $\rho$ is induced from $\mathcal{G}/\Z$ via
inflation). Here $\mathrm{Det}_{U^{\mathrm{ab}}}$ denotes the natural map
\[\mathrm{Det}_{U^{\mathrm{ab}}}: K_1(\Lambda(U^{\mathrm{ab}}))\to \mathrm{Hom}_{G_{\mathbb{Q}_p}, \, R(\Gamma_U)} (R(U^{\mathrm{ab}}),
\Lambda_{\overline{\mathbb{Z}_p}}{(\Gamma_U)}^{\times})\] with $\Gamma_U$ the image of $U$ under the fixed
projection $\mathcal{G}\twoheadrightarrow \Gamma$ and we use the embedding
$\Lambda_{\overline{\mathbb{Z}_p}}{(\Gamma_U)}\subseteq
\Lambda_{\overline{\mathbb{Z}_p}}{(\Gamma)}$ in order to consider the values of
$\mathrm{Det}_{U^{\mathrm{ab}}}$ in $\Lambda_{\overline{\mathbb{Z}_p}}(\Gamma)^\times.$

\begin{lemma}
The map $\mathrm{BrInd}$ is well-defined and $ \mathrm{BrInd}((x_U)_U)$ is $G_{\mathbb{Q}_p}$- and
$R(\Gamma)$-invariant, where these action are recalled at the end of this section.
\end{lemma}

\begin{proof}
Assume first that $\rho$ arises by inflation from $\tilde{\rho}\in \mathrm{Irr}(\mathcal{G}/\Z).$
First we check that the defining term $\prod_U \mathrm{Det}_{U^{\mathrm{ab}}}(x_U)(\chi_U)^{n_U}$
of
  $\mathrm{BrInd}((x_U)_U)(\rho)$ is independent of the choice of representatives $(n_U
\chi_U)_{U\in S(\mathcal{G},\mathcal{Z})}$ for $a_\mathcal{G}( {\rho})=a_\mathcal{G}(\tilde{\rho})$
 (using our lax notational convention): For any conjugate $g(n_U\chi_U)_{U\in
S(\mathcal{G},\mathcal{Z})}=(n_{gUg^{-1}}\chi_{gUg^{-1}}(g\cdot g^{-1}))_{U\in
S(\mathcal{G},\mathcal{Z})} $   we have by the $\mathcal{G}$-invariance of $(x_U)_U$
\begin{align*}
&\prod_U \mathrm{Det}_{U^{\mathrm{ab}}}(x_U)(\chi_{gUg^{-1}}(g\cdot g^{-1}))^{n_{gUg^{-1}}}\\
& \quad\quad\quad\quad\quad\quad\quad\quad=\prod_U
\mathrm{Det}_{U^{\mathrm{ab}}}(g^{-1}x_{gUg^{-1}}g)(\chi_{gUg^{-1}}(g\cdot g^{-1}))^{n_{gUg^{-1}}}\\
&\quad\quad\quad\quad\quad\quad\quad\quad=\prod_U\chi_{gUg^{-1}}(g(g^{-1}x_{gUg^{-1}}g)g^{-1})^{n_{gUg^{-1}}}\\
&\quad\quad\quad\quad\quad\quad\quad\quad=\prod_{gUg^{-1}}\chi_{gUg^{-1}}( x_{gUg^{-1}} )^{n_{gUg^{-1}}}\\
&\quad\quad\quad\quad\quad\quad\quad\quad=\prod_U
\mathrm{Det}_{U^{\mathrm{ab}}}(x_U)(\chi_U)^{n_U}.
\end{align*}

Next we show that our (partial) definition sofar is invariant under characters of the kind
$\chi:\mathcal{G}\twoheadrightarrow \Gamma/\Z \to \overline{\mathbb{Z}_p}^\times.$ Indeed,
\begin{align*}
 \chi \cdot\prod_U \mathrm{Det}_{U^{\mathrm{ab}}}(x_U)(\chi_U)^{n_U}
&  =\prod_U \mathrm{Res}^G_U(\chi)\cdot\mathrm{Det}_{U^{\mathrm{ab}}}(x_U)(\chi_U)^{n_U}\\
&  =\prod_U \mathrm{Det}_{U^{\mathrm{ab}}}(x_U)(\mathrm{Res}^G_U(\chi)\chi_U)^{n_U}\\
\end{align*}
by the $R(\Gamma_U)$-invariance of $\mathrm{Det}_{U^{\mathrm{ab}}}(x_U).$ But by Lemma
\ref{twist-invariance}  below $(n_U\mathrm{Res}^G_U(\chi)\chi_U)$ represents
$a_\mathcal{G}(\chi\rho),$ whence we have shown that
\[\chi \cdot\mathrm{BrInd}((x_U)_U)(\rho)=\mathrm{BrInd}((x_U)_U)(\chi\rho)\]
as claimed.

 If $\rho\in
\mathrm{Irr}(\mathcal{G})  $ is arbitrary, using \eqref{surj} we choose some
$\chi':\mathcal{G}\twoheadrightarrow \Gamma \to \overline{\mathbb{Z}_p}^\times$ such that $
\chi'\otimes\rho$ is inflated from $\mathcal{G}/\Z.$ We then define
\[\mathrm{BrInd}((x_U)_U)(\rho):=(\chi')^{-1}\cdot\mathrm{BrInd}((x_U)_U)(\chi'\otimes\rho). \] If
$\chi''$ is a second such character, we conclude that $\chi:=(\chi')^{-1}\chi''$ comes from a
character of $\mathcal{G}/\Z.$ Using the invariance of $\mathrm{BrInd}((x_U)_U) $ with respect to
such characters as shown above we conclude   that $\mathrm{BrInd}((x_U)_U)$ is well-defined on
arbitrary irreducible representations of $\mathcal{G}$ and hence extends to a homomorphism on
$R(\mathcal{G})$ being $R(\Gamma)$-invariant by construction. Finally the
$G_{\mathbb{Q}_p}$-invariance follows from the $G_{\mathbb{Q}_p}$-invariance of all the
$\mathrm{Det}_{U^{\mathrm{ab}}}(x_U)$ and clearly $\mathrm{BrInd}$ is a homomorphism.
\end{proof}

For a finite group $G$ let $R_+(G)$ denote the free abelian group on the $G$-conjugacy classes of
characters $\phi:U\to\overline{\mathbb{Z}_p}^\times,$ where $U$ is any subgroup of $G.$ We shall
write $(U,\phi)^G$ for the $G$-conjugacy class of $(U,\phi).$ We clearly have a natural isomorphism
\[ R_+(G)\cong (\oplus_{U } R(U^{\mathrm{ab}} ))_{G},\]
where $U$ runs through all subgroups of $G$ and usually explicit Brauer induction is defined in
terms of a section $a_G$ of
\[ R_+(G)\twoheadrightarrow R(G),\] i.e.,\
\[a_G(\rho)=\sum_{(U,\phi)^G}\alpha_{(U,\phi)^G}(\rho) (U,\phi)^G\]
for some integers $\alpha_{(U,\phi)^G}(\rho)$. Note that there is a natural action of any character
$\chi:G\to\overline{\mathbb{Z}_p}^\times$ on $R_+(G),$ sending $(U,\phi)^G$ to
$\chi\cdot(U,\phi)^G:=(U, \mathrm{Res}^G_U(\chi)\phi)^G.$

\begin{lemma}\label{twist-invariance}
For every $\rho\in R(G)$ we have $\alpha_{(U,\mathrm{Res}^G_U(\chi)\phi)^G}(\chi\rho)=\alpha_{(U,
\phi)^G}( \rho),$ i.e.\ $a_G(\chi\rho)=\chi\cdot a_G(\rho).$
\end{lemma}

\begin{proof}
According to \cite[thm.\ 2.3.15]{snaith} the explicit formula for $a_G$ is given by
\[\alpha_{(U,\phi)^G}(\rho)=\frac{|U|}{|G|}\sum_{\substack{(U',\phi')\in (U,\phi)^G,\\  (U',\phi')\subseteq(U'',\phi'')}} \mu_{(U',\phi'),(U'',\phi'')}<\phi'',\mathrm{Res}^G_{U''}(\rho)>,\]
where $\mu$ denotes the M\"{o}bius function on the partially ordered set $\mathcal{M}_G$ consisting of
the characters on subgroups $(U,\phi)$ of $G$ with the $(U,\phi)\subseteq (U',\phi')$ if and only
if $U\subseteq U'$ and $\mathrm{Res}^{U'}_{U}(\phi')=\phi$ while $<\; ,\;>$ denotes the usual Schur
inner product (loc.\ cit., def.\ 1.2.7). Since sending $(U,\phi)$ to
$(U,\mathrm{Res}^G_U(\chi)\phi)$ induces an $G$-equivariant isomorphism $\chi:\mathcal{M}_G\to
\mathcal{M}_G,$ of partially ordered sets, we obtain that
\[\mu_{\chi(U',\phi'),\chi(U'',\phi'')}=\mu_{(U',\phi'),(U'',\phi'')}.\] Taking into account that $<\mathrm{Res}^G_U(\chi)\phi'',\mathrm{Res}^G_{U''}(\chi\rho)>=< \phi'',\mathrm{Res}^G_{U''}( \rho)>$ it follows that
\[\alpha_{(U,\mathrm{Res}^G_U(\chi)\phi)^G}(\rho)=\frac{|U|}{|G|}\sum_{\substack{(U',\phi')\in (U,\phi)^G,\\  (U',\phi')\subseteq(U'',\phi'')}} \mu_{\chi(U',\phi'),\chi(U'',\phi'')}<\mathrm{Res}^G_U(\chi)\phi'',\mathrm{Res}^G_{U''}(\chi\rho)>,\]
equals $\alpha_{(U,\phi)^G}(\rho).$
\end{proof}

By construction, the definition of $\theta$ and the functorial behaviour of $\mathrm{Det}$ with
respect to norm maps and induction (\cite[Lem.\ 9]{RWII}):
\[\mathrm{Det}_\mathcal{G}(\lambda)(\mathrm{Ind}^\mathcal{G}_U\chi)
=\mathrm{Det}_U(N^\mathcal{G}_U(\lambda))( \chi)=
\mathrm{Det}_{U^{\mathrm{ab}}}(\theta_U(\lambda))( \chi)\]we have thus a commutative diagram
\[\xymatrix{
  K_1(\Lambda(\mathcal{G}) \ar[r]^{\mathrm{Det}_\mathcal{G}} \ar[d]_{\theta}
                    & {\mathrm{Hom}_{G_{\mathbb{Q}_p}, \, R(\Gamma)}} (R(\mathcal{G}), \Lambda_{\overline{\mathbb{Z}_p}}(\Gamma)^{\times})     \\
                 {\left(\prod_{U\in S(\mathcal{G},\mathcal{Z})}  \Lambda(U^{\mathrm{ab}})^\times \right)^\mathcal{G}} \ar[r]_(0.35){\prod
                \mathrm{Det}_{U^{\mathrm{ab}}}}
                \ar[ur]^{\mathrm{BrInd}} & {\left(\prod_{U\in S(\mathcal{G},\mathcal{Z})}\mathrm{Hom}_{G_{\mathbb{Q}_p}, \, R(\Gamma_U)}(R(U^{\mathrm{ab}}),
                \Lambda_{\overline{\mathbb{Z}_p}}(\Gamma_U)^{\times})\right)^\mathcal{G}}\ar[u]_{\mathrm{brInd}},
                }\]
where brInd is defined in an analogous way as BrInd.

While all of Kakde's congruences are among the tuples $(x_U)_U\in\prod_{U\in
S(\mathcal{G},\mathcal{Z})}  \Lambda(U^{\mathrm{ab}})^\times$ in the work of Ritter and Weiss some
are  in a similar product (as speculated below, but Ritter and Weiss do not formalise this) and
others are expressed within their Hom-description (over all with the aim to construct an integral
logarithm).

Now we are going to describe the setting of Ritter and Weiss' approach in more detail. First we set
$  Q_{\overline{\mathbb{Q}_p}} (\Gamma) := \mathrm{Quot} (\Lambda_{\overline{\mathbb{Z}_p}}
[[\Gamma]])$ and note that the homomorphism
   \[\mathrm{Det} ( \,\,\, ) (\chi) : K_1 (\Lambda(\mathcal{G})_{S^*}) \longrightarrow
   Q_L(\Gamma)^{\times}\subseteq Q_{\overline{\mathbb{Q}_p}} (\Gamma)\]
in [RW2, {\S}3] defined using the reduced norm
   $K_1(\Lambda(\mathcal{G})_{S^*} )\stackrel{nr}{\longrightarrow}
   Z(\Lambda(\mathcal{G})_{S^*})^{\times} $\\
   coincides with $\Phi_{\chi}$ of \cite{cfksv}   by \cite[lem.\ 3.1]{burns-MC}   and defines a homomorphism
  $$
   K_1 (\Lambda(\mathcal{G})_{S^*}) \stackrel{\mathrm{Det}}{\longrightarrow}
   {\mathrm{Hom}}_{G_{{\mathbb{Q}}_p}, R(\Gamma)} (R(\mathcal{G}), Q_{\overline{\mathbb{Q}_p}}{(\Gamma)}^{\times})
   $$
   into the group of $G_{{\mathbb{Q}}_p} $- and $R(\Gamma)$-invariant
   homomorphism, where:
   \begin{enumerate}
   \item[a)]
   $\sigma \in G_{\mathbb{Q}_p}$ acts coefficientwise on $Q_{\overline{\mathbb{Q}_p}}(\Gamma)$
   and $R(\mathcal{G})$,
   \item[b)]
   $\varphi \in R(\Gamma)$ acts on $R({\mathcal{G}})$
   via the tensor product $\chi \longmapsto \mathrm{infl}_\Gamma^\mathcal{G}(\varphi) \chi$ and on
   $Q_{\overline{\mathbb{Q}_p}}(\Gamma)$ via the homomorphism induced by twisting
   $$
   \begin{array}{rl}
   Tw_{\varphi}: \Lambda_{\overline{\mathbb{Z}_p}}(\Gamma) & \longrightarrow
   \Lambda_{\overline{\mathbb{Z}_p}}(\Gamma)\\
   \gamma & \longmapsto \varphi(\gamma)\gamma.
   \end{array}
   $$
      \end{enumerate}

In particular, from RW's construction { (see   Rem.\  E  in [RW2,
   {\S}4])}:
   $${\rm{ker}} {\rm{Det}}  = {\rm{ker}}{\rm{(nr}}_{\Lambda(\mathcal{G})_{S^*}}) = {\rm{SK}}_1
   (\Lambda(\mathcal{G})_{S^*}) .$$

\bn
   By \cite[Lem.\ 9]{RWII} Det behaves functorially with respect to open
   subgroups $U \subseteq {\mathcal{G}}$ and the map included by
   Induction $({\rm{Ind}}_{\mathcal{G}}^U)^*$, as well as to factor
   groups ${\mathcal{G}} \twoheadrightarrow
   {\overline{\mathcal{G}}}$ and the map induced by inflation
   $({\rm{infl}} \frac{\mathcal{G}}{\mathcal{G}})^*$ on the
   Hom-description.\\

   \section{The interpolation property versus ${L_{F_{\infty}/F}}$}

   By \[L_{F_{\infty}/F} (\chi) : = {\mathcal{L}}_{\chi, \Sigma} =
   \frac{G_{\chi, \Sigma} (\gamma - 1)}{H_{\chi} (\gamma - 1)}
   \in Q_{\overline{\mathbb{Q}_p}}(\Gamma)^{\times} \]
   we denote the $(\Sigma$-truncated) $p$-adic Artin $L$-function for
   $F^{cyc}/F$ attached to $\chi \in R({\mathcal{G}})$, for $\gamma$ a fixed
   topological generator of $\Gamma$ [RW2, {\S}4].
   Then $L_{F_{\infty}/F} $ belongs to $ {\rm{Hom}}_{G_{Q_p}, R(\Gamma)}
   (R({\mathcal{G}}), Q_{\overline{\mathbb{Q}_p}} (\Gamma)^{\times})$ and is independent of choice of
   $\gamma$ (loc.\ cit.\ prop.\ 11).

   Applying the extended localised augmentation map $\varphi': Q_{\overline{\mathbb{Q}_p}} (\Gamma)^{\times} \longrightarrow {\overline{\mathbb{Q}_p}} \cup
   \{\infty\},$ which is induced by sending any $\gamma\in \Gamma$ to $1,$ (see before \cite[thm.\ 2.4]{suja-workshopms}) one sees that giving $L_{F_{\infty}/F}$ is morally the same as
   requiring the usual interpolation property:
   $$
   {\rm{Det}}(\zeta) = L_{F_{\infty}/F}
   $$
   for some $\zeta \in K_1 (\Lambda ({\mathcal{G}})_{S^*})$
   corresponds to
   $$
   \zeta(\chi) = \varphi' \left({\rm{Det}} (\zeta)(\chi)\right) = \varphi'
   (L_{F_{\infty}/F} (\chi)) = L_{\Sigma} (\chi, 1)
   $$
   Extending this interpolation property also to the cyclotomic
   character $\kappa$, it even determines $L_{F_{\infty}/F}$
   uniquely.    In fact,  Ritter and Weiss show that
   $$
   L_{F_{\infty}/F} \in {\rm{Hom}}_{G_{\mathcal{O}_p},R(\Gamma)}^{(1)}
   (R({\mathcal{G}}), B_{\overline{\mathbb{Z}_p}}(\Gamma)^{\times}),
   $$
     where $^{(1)}$ indicates that $f$ satisfies the congruence
  \begin{equation}
  \label{snaith-cong} \frac{f(\chi)^p}{\Psi f (\psi_p \chi)} \equiv 1 \,\,\,
  {\rm{mod}} \,\,\, p
  \end{equation}  with the ring
   endomorphism
 $
  \Psi: B_{\overline{\mathbb{Z}_p}}(\Gamma) \rightarrow
   B_{\overline{\mathbb{Z}_p}}(\Gamma)
 $
being induced by sending
   $\gamma$ to $\gamma^p$, while
   $\psi_p: R({\mathcal{G}}) \rightarrow R({\mathcal{G}})$
   denotes the $p^{th}$ Adams operator, i.e.,
   \[(\psi_p \chi)(g) = \chi(g^p)\]
   for any character $\chi$.
This result uses twice explicit Brauer induction: Firstly to generalise a theorem of Snaith
\cite[thm.\ 4.1.6]{snaith} saying  that the image of $\mathrm{Det}$ lies in
${\rm{Hom}}_{G_{\mathcal{O}_p},R(\Gamma)}^{(1)}
   (R({\mathcal{G}}), B_{\overline{\mathbb{Z}_p}} (\Gamma)^{\times}).$ Due to the existence
   of Serre's pseudomeasures which can be interpreted as elements in $K_1(\Lambda(\mathcal{G})_S)$ for
   $\mathcal{G}$ abelian, it follows that $L_{F_{\infty}/F}$ satisfies the generalised Snaith
   congruences. But secondly by explicit Brauer induction the values of $L_{F_{\infty}/F}$ in the general
   case can be expressed by the values of suitable abelian $L$'s, hence implying the congruences. In
   both cases it is crucial that the Brauer induction can be arranged in a compatible way with respect to
   the $p^{th}$  Adams operator, which in general does not behave well under induction.


\bn
   \begin{theorem}[{[RW3]} Thm.\ $\rm{B}_{\wedge}$ in {\S}6] For $\mathcal{G}$
   pro-$p$ we have
   $$
   {\rm{Det}} \,\, K_1 (B({\mathcal{G}})) \cap {\rm{Hom}}_{G_{\mathbb{Q}_p},
   \; R(\Gamma)} (R({\mathcal{G}}), \Lambda_{{\overline{\mathbb{Z}_p}}} (\Gamma)^{\times}) \subset
   {\rm{Det}} K_1 (\Lambda({\mathcal{G}})).
   $$
   in $   \mathrm{ Hom}_{G_{\mathbb{Q}_p}, \; R(\Gamma)} (R({\mathcal{G}})), B_{{\overline{\mathbb{Z}_p}}}
   (\Gamma)^{\times}).
   $
\end{theorem}

   This should be compared to
   $$
    {\Phi_B} \cap \prod_{U \in S ({\mathcal{G}}, {\mathcal{Z}})}
   \Lambda (U^{\mathrm{ab}})^{\times} = \Phi
   $$
   in Kakde's work and to  Burns' result \cite[thm.~6.1, rem.~6.2]{burns-MC} - the latter approach avoids the analysis of integral
logarithms.
   Thus, by the usual argument (compare ????reference into some article of the same volume to be inserted later) the Main Conjecture is equivalent to showing that
   \begin{equation}
   \label{imageDet}L_{F_{\infty}/F} \in {\rm{Det}} \,\, {K_1} (B({\mathcal{G}})).
   \end{equation}

   In order to verify the latter condition, Ritter and Weiss introduce a new integral group
   logarithm, which makes it possible to translate this multiplicative statement in a additive
   statement plus a statement about the kernel of the integral group logarithm which is referred to
   as the ``torsion'' part (since for a finite $p$-group $G$ this kernel is actually a torsion
   group) even though in this setting it may contain a torsionfree part!

   Set $TB: = B/[B, B]$ for any ring $B$, where $[B, B]$ denotes the set
   of additive commutators $ab-ba, a,b \in B$ and consider the diagram
$$
\xymatrix{ K_1(B({\mathcal{G}}))\ar[r]^{\mathbb{L}}\ar[d]_{\rm{Det}}
&TB({\mathcal{G}})\subset TB({\mathcal{G}})[\frac{1}{p}]\ar[d]^{\rm{Tr}}_{\cong}\\
{\mathrm{Hom}_{G_{{\mathbb{Q}}_p}, R(\Gamma)}^{(1)} (R({\mathcal{G}}),B_{{\overline{\mathbb{Z}_p}}}
(\Gamma)^{\times})\ar[r]^{ {\text{\bf{L}}}}}
     &{\mathrm{Hom}_{G_{{\mathbb{Q}_p}}, {R(\Gamma)}}({R}({\mathcal{G}}),
     B_{{\overline{\mathbb{Z}_p}}}(\Gamma)[\frac{1}{p}] )
     }}
$$
$$
    f \mapsto {\text{\bf{L}}}(f)(\chi) = \frac{1}{p} \log
    (\frac{f(\chi)^p}{\Psi(f(\psi_p \chi))})
$$
in which   ${\text{\bf{L}}f}$ is well-defined due to the congruence \eqref{snaith-cong} (note that
due the Galois invariance, for any $f$ in   $\mathrm{Hom}_{G_{{\mathbb{Q}}_p}, R(\Gamma)}
(R({\mathcal{G}}),B_{{\overline{\mathbb{Z}_p}}} (\Gamma)^{\times})$ the value $f(\rho)$ belongs to some
$B_{{\mathcal{O}_L}} (\Gamma)^{\times}$ for some finite extension $L$ of $\mathbb{Q}_p)$ and which defines
the integral logarithm $\mathbb{L}=\mathbb{L}_{\mathcal{G}}$ in the upper row following the
approach of Snaith in \cite{snaith}. The trace map Tr is the analogue of Det in the additive
setting, see \cite{RWIII}.

{\bf Question:} Does $\mathbb{L}$ coincide with Kakde's integral logarithm $L_B$ in \cite[\S
5]{sch-venkakde}?

As mentioned above now Ritter and Weiss  divide the condition \eqref{imageDet} into the question
whether the ``additive'' element $\mathbf{L}(L_{F_{\infty/F}})$ lies in
$\mathrm{Tr}(\mathrm{im}(\mathbb{L}))$ and under which conditions a (torsion) element in
$\ker(\mathbf{L}),$ viz the defect of $L_{F_{\infty/F}}$ not being determined by
$\mathbf{L}(L_{F_{\infty/F}}), $ is in the image of $\mathrm{Det}$? To this aim they introduce
   \[t_{F_{\infty/F}}:= Tr^{-1} {\text{\bf{L}}}(L_{F_{\infty/F}}) \in T
   B({\mathcal{G}})[\frac{1}{p}],\]
  and call it the {\it{logarithmic pseudomeasure}}.

\begin{theorem}\label{reductionlog}
$L_{F_\infty/F} \in \mathrm{Det }
    K_1(B(\mathcal{G}))$ if and only if
   \begin{enumerate}
    \item[{(i )}] $t_{F_\infty/F}\in T B(\mathcal{G})$ (integrality) and
    \item[{(ii)}] $\mathrm{ver}_U^V\ \zeta_{F_\infty^{[V,V]}/F_\infty^V}\equiv
    \zeta_{F_\infty/F_\infty^U}\;\mathrm{ mod }\; \mathrm{im}(\sigma_U^V)$ (torsion congruence)
   \end{enumerate}
   for all $U\subseteq V\subseteq\mathcal{G}$ with $U$ abelian, $[V:U]=p$
   where $\mathrm{ver}_U^V$ is induced from $\mathrm{ver}_U^V:V^{\mathrm{ab}}\to U^{\mathrm{ab}}$
   by linear extension.

\begin{remark}
Ritter and Weiss' strategy of decomposing the problem into an additive/logarithmic and torsion part
by various diagram chases should be compared with the use of the $5$-lemma in Kakde's approach with
respect to the following diagram  in \cite[between (8) and lem.\ 4.7]{sch-venkakde}
\begin{equation}\label{5termdiagram}
    \xymatrix{
       1  \ar[r] & \mu(\mathcal{O}) \times \G^{\mathrm{ab}} \ar[d]_{=} \ar[r]^-{\iota} & K'_1(\Lambda(\G)) \ar[d]_{\theta} \ar[r]^-{L} & \mathcal{O}[[\mathrm{Conj}(\G)]] \ar[d]_{\beta}^{\cong} \ar[r]^-{\omega} & \G^{\mathrm{ab}} \ar[d]_{=} \ar[r] & 1  \\
     1 \ar[r] & \mu(\mathcal{O}) \times \G^{\mathrm{ab}} \ar[r]^-{\theta\circ \iota} & \Phi \ar[r]^-{\mathcal{L}} & \Psi \ar[r]^-{\omega\circ\beta^{-1}} & \G^{\mathrm{ab}} \ar[r] & 1 .  }
\end{equation}
combined with diagram \cite[(11)]{sch-venkakde} for the $B$-situation.

\end{remark}
\end{theorem}
  \bn
  {\bf{Idea of Proof:}} (for the converse direction). If $t_{F_\infty/F}\in TB(\mathcal{G})$, then there exists $y\in K_1(B(\mathcal{G}))$ such that
 \begin{equation}\label{imageL}
t_{F_\infty/F} = \mathbb{L}(y)
 \end{equation}
   and           \begin{equation}
   \label{y}y \mapsto  \zeta_{F_\infty^{[\mathcal{G},\mathcal{G}]}/F}
     \end{equation}
  under the canonical map $\mathrm{pr}_{\mathcal{G}^{\mathrm{ab}}}^{\mathcal{G}}:
  K_1(B(\mathcal{G}))\to K_1(B(\mathcal{G}^{\mathrm{ab}}))$.
  This follows immediately from the abelian case
  $\mathbb{L}_{\mathcal{G}^{\mathrm{ab}}}(\zeta_{F_\infty^{[\mathcal{G},\mathcal{G}]}/F})
  =t_{F_\infty^{[\mathcal{G},\mathcal{G}]}/F}$ and the following commutative diagram
  \cite[lem.\   7.1 ii)]{RWIX}
    with exact rows
  $$
  \xymatrix{
  1 \ar[r] & 1 + \mathfrak{a}
  \ar[r]
  \ar@{->>}[d]
  & B(\mathcal{G})^{\times}
  \ar@{->>}[r] \ar[d]_{{\mathbb{L}}_{\mathcal{G}}} &
  B(\mathcal{G}^{\mathrm{ab}})^{\times}
  \ar[r]\ar[d]_{\mathbb{L}_{{\mathcal{G}}^{\mathrm{ab}}}} & 0\\
  0 \ar[r] & \tau({\mathfrak{a}}) \ar[r] & T B(\mathcal{G})
  \ar[r] & T B(\mathcal{G}^{\mathrm{ab}}) \ar[r] & 0\\
  & & t_{F_\infty/F} \ar@{|-{>}}[r]
 & t_{F_\infty^{[\mathcal{G},\mathcal{G}]}/F}
   }
  $$

\bn
  where $\mathfrak{a}:= \mathrm{ker}(B(\mathcal{G}) \rightarrow B(\mathcal{G}^{\mathrm{ab}}))$
  and $\tau(\mathfrak{a})$ is the image of $\mathfrak{a}$ with respect to the
  canonical map $\tau: B(\mathcal{G}) \twoheadrightarrow TB(\mathcal{G})$.
 \bn
  Setting $\omega:= \mathrm{Det}(y)^{-1}\cdot L_{F_\infty/F}$ we have by \eqref{y} that
    \begin{equation}
  \label{omega} \omega_{|R(\mathcal{G}^{\mathrm{ab}})}\equiv 1
    \end{equation}

  and by the definition of $t_{F^\infty/F}$
  $$
  {\mathbf{L}}(\omega)=0,
  $$
  whence \[\frac{\omega(\chi)^p}{\Psi\omega(\psi_p\chi)}=1,\] as $\log_{|1+p\,
  B_{{\mathcal{O}}_L}(\Gamma)}$
  is injective. Therefore $\omega(\chi)^{p^n}={\Psi}^n \omega(\psi_p^n\chi)=\Psi^n
  (\omega({\mathds{1}}))^{\chi({\mathds{1}})}=1$
  for $n$ sufficiently big such that $\psi_{p}^{n} \mathrm{R}
  (\mathcal{G}/Z) = \{\mathds{1}\}$.

  That means that $\omega$ is a torsion element. In [RW5, prop. 2.4] even
  uniqueness of $\omega$ is shown.
  We want to show that $\omega = 1$.
  Assume first that $\mathcal{G}$ contains an abelian subgroup
  $\mathcal{G}'$ of index $p$. Since then any irreducible
  representation $\chi$ of $\mathcal{G}$ is either inflated from an
  abelian character $\alpha$ of $\mathcal{G}^{\mathrm{ab}}$
  $$
  \chi = \rm{infl}_{\mathcal{G}^{\mathrm{ab}}}^{\mathcal{G}} \alpha
  $$
  or induced from an abelian character $\beta$ of $\mathcal{G}'$
  $$
  \chi = \rm{ind}_{\mathcal{G}'}^{\mathcal{G}} (\beta'),
  $$
  it suffices by \eqref{omega}  to verify that
 \begin{equation}\label{5}
\omega | \rm{Ind}_{\mathcal{G}'}^{\mathcal{G}}
  \mathcal{R}(\mathcal{G'}) \equiv 1 \;\;\; \rm{or} \;\;\;
  (\rm{Ind}_{\mathcal{G}'}^{\mathcal{G}})^* \omega \equiv 1
   \end{equation}

\bn
  But by the functoriality properties
  $ (\rm{Ind}_{\mathcal{G}'}^{\mathcal{G}})^* L_{F_{\infty}/F} =
  L_{F_{\infty}/F_{\infty}^{\mathcal{G}'}}$ and Det
  $N_{\mathcal{G}'}^{\mathcal{G}} (y) = (\mathrm{Ind}
  _{\mathcal{G}'}^{\mathcal{G}})^*
  {\mathrm{Det}} (y) $
  we obtain
  $$
  \left(\mathrm{Ind}_{\mathcal{G}'}^{\mathcal{G}}\right)^* \omega =
   \frac{\left({\mathrm{Ind}}_{\mathcal{G}'}^{\mathcal{G}} \right)^*
   L_{F_{\infty} / F}}
   {{\mathrm{Det}}\left( N_{\mathcal{G}'}^{\mathcal{G}} (y)\right)} =
   \mathrm{Det}
   \left( \frac{\zeta_{F_{{\infty} / F_{\infty}^{\mathcal{G}'}}}}
   {N_{\mathcal{G}'}^{\mathcal{G}} (y)} \right)
  $$

\bn where $N_{\mathcal{G}'}^{\mathcal{G}} :{{K}}_1({B}({\mathcal{G}})) \rightarrow
{{K}}_1({{B}}({\mathcal{G}'}))$ denotes the norm map. Since Det is injective on
${{K}}_1({B}({\mathcal{G}'}))$, we see that
\[\frac{\zeta_{{F_{\infty}} /
 {{F_{\infty}}^{\mathcal{G}'}}}}{N_{\mathcal{G}'}^{\mathcal{G}} (y)}\] is a
 torsion element.
  By the Wall-congruence ([RW3, proof of Lemma 12])
  $$
  N_{\mathcal{G}'}^{\mathcal{G}}(y) \equiv \mathrm{ver}
  _{\mathcal{G}'}^{\mathcal{G}^{\mathrm{ab}}} \left( pr_{\mathcal{G}^{\mathrm{ab}}}
  ^{\mathcal{G}} (y) \right) \mbox{ mod } \mbox{ im }
  (\sigma_{\mathcal{G}'}^{\mathcal{G}})
  $$
  (which corresponds in this special case to (M3) in Kakde's work, see \cite{sch-venkakde}: $N_{\mathcal{G}'}^{\mathcal{G}}(\theta_{\mathcal{G}'}(x)) \equiv \mathrm{ver}
  _{\mathcal{G}'}^{\mathcal{G}^{\mathrm{ab}}} \left( \theta_\mathcal{G} (x) \right) \mbox{ mod } \mbox{ im }
  (\sigma_{\mathcal{G}'}^{\mathcal{G}})$)
  it follows that by our choice of $y$ with
  $\mathrm{pr}_{{\mathcal{G}}^{\mathrm{ab}}}^{\mathcal{G}} (y)=
  \zeta_{{F_{\infty}}^{[{\mathcal{G}}, {\mathcal{G}}]}/F_{\infty}}$

$$
\begin{array}{lll}
e:=  \frac{\zeta_{{F_{\infty}}/{{F_{\infty}}^{\mathcal{G}'}}}} {N_{\mathcal{G}'}^{\mathcal{G}}(y)}
& =  \;\; \frac{\mathrm{ver}_{\mathcal{G}'}^{{\mathcal{G}}^{\mathrm{ab}}}
\left(\mathrm{pr}_{{\mathcal{G}}^{\mathrm{ab}}}^{\mathcal{G}} (y)\right)}
{N_{\mathcal{G}'}^{\mathcal{G}} (y)} & \frac{\zeta_{{F_{\infty}}/{{F_{\infty}}^{\mathcal{G}'}}}}
{\mathrm{ver}_{\mathcal{G}'}^{{\mathcal{G}}^{\mathrm{ab}}} \left(
\zeta_{{{F_{\infty}}^{[{\mathcal{G}},{\mathcal{G}}]}} / F_{\infty}}
\right)}\\
&&\\
 & \equiv
 \frac{\zeta_{{F_{\infty}}/{F_{\infty}^{\mathcal{G}'}}}}
 {\mathrm{ver}_{\mathcal{G}'}^{{\mathcal{G}}^{\mathrm{ab}}}
(\zeta_{{F_{\infty}^{[{\mathcal{G}},{\mathcal{G}}]}} / F_{\infty}})}
 & \stackrel{(ii)}{\equiv} \;\; 1 \;\; \mathrm{mod} \;\; \mathrm{im}
 (\sigma_{\mathcal{G}'}^{\mathcal{G}})
\end{array}
$$

\bn
 by our assumption. The Theorem follows from the

 \bn
 {\bf{Claim}:} If $e\in B(\mathcal{G}')^{\times}$ is torsion and satisfies
 $e\equiv1\mbox{ mod } \mathrm{im} (\sigma_{\mathcal{G}'}^\mathcal{G})$, then $e=1,$ cp.\ with
 \cite[proof of lem.\ 4.9]{sch-venkakde}.

 By  Higman's theorem we have: $e=\zeta h\equiv1\mbox{ mod }\mathrm{im}(\sigma_{\mathcal{G}'}^G)$
 with $h\in H'$ (for the group ring $B(\mathcal{G}')=B(\Gamma')[H']$,
 if $\mathcal{G}'$ decomposes as $\Gamma'\times H'$).
The augmentation map $\epsilon : {B}({\mathcal{G}'}) \rightarrow {B} (\Gamma')$ induces $ \zeta
\equiv 1 \;\mathrm{mod} \; p B({\Gamma}')$, whence
 $\zeta = 1$ and $h' -1 \in {\mathrm{im}}
 (\sigma_{\mathcal{G}}^{\mathcal{G}'})$.
 Thus $h' = 1, e = 1$ and \eqref{5} holds.\\

 By an inductive argument this argument can be extended to arbitrary
 $\mathcal{G}$, see proof of Theorem in [RW5, {\S}3].\hfill $\Box$

 \bn
 The hard part is now to inductively show that
 \begin{equation}
 \label{tintegral}t_{F_{\infty} / F} \in T {B}({\mathcal{G}})
 \end{equation}
 holds,  which requires the new M\"{o}bius-Wall congruence
 $$
 \sum_{{\mathcal{A}} \subseteq U \subseteq {\mathcal{G}}}
 \mu_{{\mathcal{G}}/{\mathcal{Z}}} (U/{\mathcal{Z}})
 {\mathrm{ver}}_{\mathcal{A}}^{U^{\mathrm{ab}}}
 (\zeta_{F_{\infty}^{[U, U]} / F_{\infty}^U}) \in \mathrm{im}
 (\sigma_{\mathcal{A}}^{\mathcal{G}})
  $$
for any abelian normal open subgroup $\mathcal{A}\trianglelefteq\mathcal{G}$ (actually for each
such one-dimensional subextension) introduced in \cite{RWIX}, more precisely in (loc.\ cit.) only a
similar relation for units in $B(\mathcal{G})$ is called M\"{o}bius-Wall congruence.
 Recall that for a finite $p$-group $G$, the M\"{o}bius-function is defined
 inductively as follows
 $$
 \begin{array}{rl}
  \mu_G{(1)} & = 1\\
  \mu_G (U)   & = - \sum_{V \subsetneq U}
  \mu_{\mathcal{G}}(V) \;\; {\rm{for}} \;\; 1 \neq U \subseteq
  {{G}}.
 \end{array}
 $$

 How this condition enters the proof will be explained in   the next section, in which  we try first
 to abstract and formalise what the methods of Ritter and Weiss actually
 prove.

 \section{The abstract setting - a reinterpretation of Ritter and Weiss' approach}

Fix a one dimensional pro-$p$-group $\G$ with projection onto $\Gamma\cong\mathbb{Z}_p,$ let $H$
denote its kernel and define the following index sets
\begin{eqnarray*}
S:=S_\G &:=&\{U \mbox{ one-dimensional subquotient of }\G\}\\ &\phantom{,}=&\{U| U\subseteq
\G/C\mbox{ open  for some } C\trianglelefteq\G \mbox{ with } C\subseteq H\}
\end{eqnarray*} and
\[S^{\mathrm{ab}}:=S_\G^{\mathrm{ab}}=\{U\in S_\G| U\mbox{ abelian}\}.\] Note that for any $U\in S$ the quotient
$U^{\mathrm{ab}}$ also belongs to $S$ and in particular to $S^{\mathrm{ab}}.$ Define
$\widetilde{\Phi}_B$ to be the subgroup of $\prod_{U\in S^{\mathrm{ab}}_\G} B(U)^\times$ consisting
of those $(\lambda_U)_{U\in S^{\mathrm{ab}}_\G}$ satisfying the following conditions:

\bn (RW1) \parbox[t]{13cm}{For every surjection $U\twoheadrightarrow V$ in $S^{\mathrm{ab}}$ we
have
\[\mathrm{pr}^U_V\lambda_U=\lambda_V,\] where $\mathrm{pr}^U_V: K_1(B(U))\rightarrow K_1(B(V))$ is
the natural map induced by the projection, and for  every inclusion $V\subseteq U$ in
$S^{\mathrm{ab}}$ we have
\[N^U_V\lambda_U=\lambda_V,\] where $N^U_V$ denotes the norm map.}

 \bn
 (RW2) \parbox[t]{13cm}{For all $U\in S_\G$ the (sub-)tuple $(\lambda_{V^{\mathrm{ab}}})_{V\subseteq U}$ is
 $U$-invariant.}

 \bn
 (RW3) \parbox[t]{13cm}{(M\"{o}bius-Wall congruence) For all $U\in S_\G$ and all abelian normal open subgroups $\A\trianglelefteq U$  we have
 \[\sum_{\A\subseteq V\subseteq U} \mu_{U/\A}(V/\A) \mathrm{ver}^{V }_\A(\lambda_{V^{\mathrm{ab}}})\in \sigma^U_\A(B(\A)).\]  }

In particular the  torsion congruence

\bn (RW3a) \parbox[t]{13cm}{(Torsion congruence) For all $U\in S_\G$ and all abelian normal open
subgroups $\A\trianglelefteq U$  of index $p$ we have
 \[  \mathrm{ver}^{U}_\A(\lambda_{U^{\mathrm{ab}} })-\lambda_\A \in \sigma^U_\A(B(\A)).\]  }

Actually Ritter and Weiss show that (RW3) holds for every tuple which arises from an element
$\vartheta\in K_1(B(\G)),$ see \cite[thm.\ 2]{RWIX}. Strictly speaking, they only call the relation
for such $\vartheta$ M\"{o}bius-Wall congruence, but we extend this notation to tuples in
$\widetilde{\Phi}_B.$  The proof generalises Wall's proof of (RW3a) by analysing the
Leibniz-formula for determinants; for combinatorial reasons the M\"{o}bius function shows up. Is it by
chance that this or a similar M\"{o}bius function also shows up in the explicit formula of Brauer
induction? While the proof of (RW3) is rather tedious it is straightforward to check (RW1) and
(RW2).

Using explicit Brauer induction (as at the beginning of section \ref{homdescription}), for any
$U\in S_\G$ one can assign to  a tuple $(\lambda_V)\in\widetilde{\Phi}_B,$ or rather to its
sub-tuple $(\lambda_V)_{V\subseteq U},$  elements
\[\Xi_U\in \mathrm{Hom}_{G_{{\mathbb{Q}}_p}, R(\Gamma_U)}^{(1)} (R({U}),B_{{\mathcal{O}}_L}
(\Gamma_U)^{\times})\] and
\[t_U:=Tr^{-1}(\mathbf{L}_U(\Xi_U))\in TB(U)[\frac{1}{p}],\]
such that $t_U=\mathbb{L}_U(\lambda_U)\in TB(U)$ for all $U\in S^{\mathrm{ab}}_\G.$ Indeed, for
every one-dimensional character $\rho$ of $U$ (which is trivial on some central subgroup
$\mathcal{Z}_U$ of $U$) $a_U(\rho)$ is represented by $\rho$ itsself under explicit Brauer
induction by \cite{boltje}. Finally we require

\bn (RW4) \parbox[t]{13cm}{ For any $U\in S_\G,$ the definition of $\Xi_U$ does not depend on the
above chosen way by explicit Brauer induction, i.e.\
\[\Xi_U(\rho)=\prod_{V} \mathrm{Det}_{V^{\mathrm{ab}}}(\lambda_{V^{\mathrm{ab}}})(\chi_{V})^{n_V},\] whenever  \[\rho=\sum n_V
\mathrm{Ind}_U^V\chi_V
\] in $R(U)$  for certain subgroups $V\subseteq U,$  one-dimensional representations $\chi_V$ of $V$ and (finitely many nonzero) integers $n_V.$}

This conditions looks a little weird, but whenever one is interested in $p$-adic $L$-functions, it
is completely harmless, as it is always satisfied by the usual behaviour of $L$-functions under
induction.

\bn {\bf Questions:} Do conditions (RW1-3) imply already (RW4)? Is there a way of proving the next
Lemma \ref{functorial} without requiring (RW4)?

All what one needs to extend Ritter and Weiss proof of property \eqref{tintegral}  to Theorem
\ref{t_G} below are the following functoriality properties.

\begin{lemma}\label{functorial}
 $\mathrm{pr}^U_V(\Xi_U)=\Xi_V$ and $N^U_V(\Xi_U)=\Xi_V$
  as well as similarly   $\mathrm{pr}^U_V(t_U)=t_V$ and $\tilde{\mathrm{Tr}}^U_V(t_U)=t_V$ for all possible $U,V\in S_\G,$
  where the modified trace $\tilde{\mathrm{Tr}}^U_V$ is introduced in \eqref{modtrace} below.
\end{lemma}

\begin{proof}
While it is well-known that explicit Brauer induction (given by $a_G$ as above) behaves well under
inflation, it does not behave well under induction. Therefore we need at present condition (RW4)
here to prove the correct behaviour under the norm: For $V\subseteq U,$ let \[\rho=\sum_{W\subseteq
V} n_W \mathrm{Ind}_V^W \chi_W\] in $R(V).$ Then, by the transitivity of induction  we obtain
\[\mathrm{Ind}(\rho)=\sum_{W\subseteq V} n_W
\mathrm{Ind}_U^W \chi_W\] in $R(U).$ By the definition of the norm on the Hom-description we thus
have
\begin{eqnarray*}
(N^U_V\Xi_U)(\rho)&=&\Xi_U(\mathrm{Ind}(\rho))\\
 &=& \prod_{W\subseteq V} \mathrm{Det}_{W^{\mathrm{ab}}}(\lambda_{W^{\mathrm{ab}}})(\chi_W)^{n_W}\\
 &=& \Xi_V(\rho).
\end{eqnarray*}
For $pr:U\twoheadrightarrow V$ we obtain
\[\mathrm{infl}^U_V(\rho)=\sum_{W\subseteq V} n_W \mathrm{Ind}^{W'}_U(\mathrm{infl}^{W'}_W\chi_W),\]
where $W':=pr^{-1}(W)$ is the full preimage of $W$ under $pr.$ Hence
\begin{eqnarray*}
 \mathrm{pr}^U_V(\Xi_U)(\rho)&=&\Xi_U(\mathrm{infl}^U_V(\rho)) \\
   &=&\prod_{W\subseteq V} \mathrm{Det}_{W'^{\mathrm{ab}}}(\lambda_{W'^{\mathrm{ab}}})(\mathrm{infl}^{W'}_W\chi_W)^{n_W}\\
   &=&\prod_{W\subseteq V} \mathrm{Det}_{W^{\mathrm{ab}}}(\mathrm{pr}^{{W'}^{\mathrm{ab}}}_{W^{\mathrm{ab}}}({\lambda_{{W'}^{\mathrm{ab}}}}))( \chi_W)^{n_W}\\
   &=&\prod_{W\subseteq V} \mathrm{Det}_{W^{\mathrm{ab}}}( \lambda_{W^{\mathrm{ab}}})( \chi_W)^{n_W}\\
   &=& \Xi_V(\rho).
\end{eqnarray*}
(the last property does not require (RW4)!). The corresponding statements for $t_U$ follow from the
functorial properties of $Tr^{-1}\circ\mathbf{L}_U$, see \cite[lem.\ 7.2]{RWIX}.
\end{proof}

\begin{theorem}
For every $(\lambda_U)_U\in\widetilde{\Phi}_B$ there exist $\lambda_\G\in K_1(B(\G))$ such that
\[\mathrm{Det}_\G(\lambda_\G)=\Lambda_\G.\] 
\end{theorem}

{\bf Question:} How big is the kernel of $\mathrm{Det}_\G?$

In order to obtain \eqref{imageDet} from this Theorem we just have to observe that for the tuple
$(\lambda_U)_{U\in S^{\mathrm{ab}}}$ consisting of the abelian pseudomeasures
$\lambda_U=\lambda_{F_\infty^C/F_\infty^V}$ if $U=V/C\in S^{\mathrm{ab}}$ for some subgroup
$V\subseteq\mathcal{G}=G(F_\infty/F)$ and some normal subgroup $C\trianglelefteq G$ contained in
$H$ the associated element $\Xi_G$ equals $L,$ supposed of course that $(\lambda_U)_{U\in
S^{\mathrm{ab}}}$ belongs to $\widetilde{\Phi}_B.$ While (RW1), (RW2) and (RW4) are well-known
properties, (RW3) forms a completely new property which is proved by Ritter and Weiss using the
$q$-expansion principle of Deligne-Ribet.

Now we are going to prove the Theorem:  using (RW3a) and the analogue of Theorem
\ref{reductionlog}, which also can be proved in this general setting, it suffices to prove
\begin{theorem}\label{t_G}
In the situation of the above Theorem we have \[t_\G\in TB(\G).\]
\end{theorem}

\bn {\bf Idea of proof:} Actually we shall show that $t_U\in TB(U)$ for all $U\in S_\G$ by
contraposition (alternatively one could formulate the argument using induction): If this statement
is false we firstly may consider among the counterexamples those $U$ for which the order of $[U,U]$
is minimal (this order must be different from one as the claim of the proposition holds for all
abelian $U$). Among those $U$ we may assume that the order of $[U:Z(U)]$ is minimal, where $Z(U)$
denotes the centre of $U.$ Without loss of generality we may and do assume that these minimality
conditions already hold for $\G$ itself. Then we choose  a central element $c\in [\G,\G]$ of order
$p$ and set $C:=<c>.$ Furthermore we choose a maximal (with respect to inclusion) abelian normal
subgroup $\A$ of $\G,$ which then automatically contains $C.$ Also we fix a central subgroup
$\Z\cong\mathbb{Z}_p$ of $\G$ contained in $\A.$

In order to arrive at a contradiction we will use  the following lemma of Ritter and Weiss, in
which  $\tilde{\mathrm{Tr}}^U_V $ for $V\subseteq U$ denotes the modified trace map defined by
Ritter and Weiss (extending the definition in \cite{oliver-book} to their Iwasawa theoretic
Hom-description) to make the following diagram, in which $N^U_V$ denotes the norm map, commutative
\begin{equation}
\label{modtrace}\xymatrix{
  K_1(B(U )) \ar[d]_{N^U_V} \ar[r]^{\mathbb{L}_U } & TB(U) \ar[d]^{\tilde{\mathrm{Tr}}^U_V } \\
  K_1(B(V))\ar[r]^{ \mathbb{L}_V} & TB(V) .  }\end{equation}
Note that here we encounter another significant difference in comparison to Kakde's approach, who
uses in a similar situation the original trace map $\mathrm{Tr}^V_U,$ but then has to take a
modified integral logarithm map $\widetilde{\mathbb{L}}_V$ in order to obtain a similar commutative
diagram (which then induces the diagram in \cite[prop.\ 4.4 and (11)]{sch-venkakde} involving
$\Phi$ and $\Psi$).

\begin{lemma}\label{injectivemodtrace}
Defining $\tilde{\beta}=(\tilde{\beta}_U)$ by
$\tilde{\beta}_U=\mathrm{pr}^U_{U^{\mathrm{ab}}}\tilde{\mathrm{Tr}}^\G_U$ for $U$ in
$S(\G,\A):=\{U\in S_\G|\A\subseteq U\}$ we obtain a commutative diagram of {\em injective}
homomorphisms
\[\xymatrix{
  TB(\G) \ar[d]_{ } \ar[r]^(0.4){\tilde{\beta}} & {\prod_{U\in S(\G,\A)} B(U^{\mathrm{ab}})} \ar[d]^{ } \\
  TB(\G)[\frac{1}{p}] \ar[r]^(0.35){\tilde{\beta}[\frac{1}{p}]} & {\prod_{U\in S(\G,\A)} B(U^{\mathrm{ab}})[\frac{1}{p}].}    }\]
\end{lemma}

For the desired contradiction  it thus suffices to show that there exist a $t\in TB(\G)$ such that
$\tilde{\beta}(t)=(t_{U^{\mathrm{ab}}})_{U\in S(\G,\A)},$ because then $t=t_\G$ as $t_\G$ is mapped
to the same tuple under $\tilde{\beta}$ by Lemma \ref{functorial}. We first will search for a $t$
such that
\begin{equation}
\label{t} t\mapsto t_\A,\; t_{\G^{\mathrm{ab}}}.
\end{equation}

Setting $\bar{\G}:=\G/C$ and $\bar{\A}:=\A/C$ we define the ideals $\mathfrak{a}$ and
$\mathfrak{b}$ by exactness of rows in the following commutative diagram
\[\xymatrix{
  0   \ar[r]^{ } & {\mathfrak{a}}  \ar[r]^{ } & B(\G)   \ar[r]^{ } & B(\bar{\G})   \ar[r]^{ } & 0   \\
  0 \ar[r]^{ } &  {\mathfrak{b}} \ar[r]^{ }\ar[u]_{ } & B(\A) \ar[r]^{ }\ar[u]_{ } & B(\bar{\A}) \ar[r]^{ }\ar[u]_{ } & 0.   }\]

Ritter and Weiss show that one has a commutative diagram

\[\xymatrix{
  0   \ar[r]^{ } & {\tau\mathfrak{a}}  \ar[r]^{ }\ar@{->>}[d]_{ } & TB(\G)  \ar[dd]_{\tilde{\mathrm{Tr}}^\G_\A } \ar[r]^{ } & TB(\bar{\G}) \ar[dd]_{\tilde{\mathrm{Tr}}^{\bar{\G}}_{\bar{\A}} }  \ar[r]^{ } & 0
  \\ & \sigma^\G_\A(\mathfrak{b}) \ar@{^(->}[d] \\
  0 \ar[r]^{ } &  {\mathfrak{b}} \ar[r]^{ } & B(\A) \ar[r]^{ }  & B(\bar{\A}) \ar[r]^{ }  & 0,   }\]

By the minimality of $[\G,\G]$ we know that $t_{\bar{\G}}$ belongs to $TB(\bar{\G}).$ We choose any
lift $t=\mathbb{L}_\G(\vartheta) \in TB(\G)$ of $t_{\bar{\G}}$  in the image of $\mathbb{L}_\G$
(which is possible by the same reasoning  as in the proof of Theorem \ref{reductionlog} with the
additional property that $\mathrm{pr}^\A_{\bar{\A}}N^\G_\A(\vartheta)=\lambda_{\bar{\A}}$). Then
the condition (RW3) for the tuple $(\lambda_U)_U$ reads
\[\sum_{\A\subseteq V\subseteq \G} \mu_{\G/\A}(V/\A) \mathrm{ver}^{V }_\A(\lambda_{V^{\mathrm{ab}}})\in \sigma^\G_\A(B(\A))\]
while one easily shows - using the fact that $N^\G_\A(\vartheta)\lambda_\A^{-1}$ lies in the kernel
of $B(\A)^\times\to B(\bar{\A})^\times$ - that for the tuple
$(\mathrm{pr}^U_{U^{\mathrm{ab}}}N^\G_U(\vartheta))_U$ induced by $\vartheta$ it reads
\[N^\G_\A(\vartheta)+\sum_{\A\subsetneqq V\subseteq \G} \mu_{\G/\A}(V/\A) \mathrm{ver}^{V }_\A(\lambda_{V^{\mathrm{ab}}})\in \sigma^\G_\A(B(\A))\]
since \[\mathrm{ver}^{V }_\A(\mathrm{pr}^V_{V^{\mathrm{ab}}}N^\G_V(\vartheta))= \mathrm{ver}^{V
}_\A(\lambda_{V^{\mathrm{ab}}}) \mbox{ for } V\neq\A.\]

This implies that \[N^\G_\A(\vartheta)\lambda_\A^{-1}\in 1+\sigma^\G_\A(B(\A)),\] from which Ritter
and Weiss derive that
\[x:=\tilde{\mathrm{Tr}}^\G_\A(t)-t_\A=\mathbb{L}_\G(N^\G_\A(\vartheta)\lambda_\A^{-1})\in \mathfrak{b}^\G\cap \sigma^\G_\A(B(\A))=\sigma^\G_\A(\mathfrak{b})\]
lies in the image of $\tau(\mathfrak{a})$ under $\tilde{\mathrm{Tr}}^\G_\A.$ Upon modifying $t$
accordingly,  we may and do thus assume that
\[\tilde{\mathrm{Tr}}^\G_\A(t)=t_\A.\] as claimed above.
%

In fact, it follows easily that then   $t$ is also mapped to  $t_{U^{\mathrm{ab}}}$ under
$\tilde{\beta}_U$ for all $U$ such that $c$ is contained in $[U,U].$ Otherwise we have
$C\cap[U,U]=1,$ i.e., the order of $[U,U]$ is strictly smaller than that of $[\G,\G],$ whence
\[t_U\in TB(U)\]
by assumption. We set \[x_U:=x_U(t):=\tilde{\mathrm{Tr}}^\G_U(t)-t_U\in TB(U).\] By our
contraposition the support, i.e., the set \[\mathrm{supp}(t):=\{U\in S(\G,\A)|C\cap[U,U]=1,\;
x_U(t)\neq 0\},\] is non-empty for all choices of $t$ satisfying \eqref{t}. Let $t$ be such a
choice with the order of $\mathrm{supp}(t)$ being minimal and let $U\in \mathrm{supp}(t)$ have
minimal $[U:\A].$ Explicit calculations with $\tilde{\mathrm{Tr}}$ using \cite[lem.\ 5.1]{RWIX},
the minimality of $[\G:Z(\G)]$ and the uniqueness of expressing elements in $TB(U)$ as linear
combinations over $B(\Z)$ similarly as for the proof of $\Psi\to\Psi_{cyc}$  being injective in
\cite[lem.\ 3.5]{sch-venkakde} show firstly that $U/\A$ must be cyclic and secondly that in this
cyclic case one can modify $t$ to an element $t'$ (using \cite[\S 6, claim 6.A]{RWIX}) with
strictly smaller support than $t,$ a contradiction.\hfill$\Box$

At the end we want to investigate whether for $\Lambda(\mathcal{G})$ itself instead of
$B(\mathcal{G})$ the analogous group $\widetilde{\Phi}$ in terms of $\Lambda(-)$ satisfying again
(RW1-4) describes, i.e., equals the image of $K_1(\Lambda(\mathcal{G}))$.
 From Kakde's result we know that his condition (M4) (together with (M1-M3) at
least) imply the M\"{o}bius-Wall congruence (RW3), even though there does not seem to be a direct link
between them. For example (M4) involves only cyclic subgroups while (RW3) ranges over all subgroups
among $\A$ and $\G.$ In fact we will now prove  the converse, viz that (RW1-4) also implies (M1-4).

\begin{theorem}\label{image}
\[K'_1(\Lambda(\G)):= K_1(\Lambda(\G))/SK_1(\Lambda(\G))\cong\tilde{\Phi}.\]
\end{theorem}

\begin{proof} Consider the commutative diagram with exact rows
\[\xymatrix{
  0   \ar[r]^{ } & {\mu\times\G^{\mathrm{ab}}} \ar[d]_{\iota } \ar[r]^{ } & K_1'(\Lambda(\G)) \ar[d]_{\tilde{\theta}} \ar[r]^{\mathbb{L}} & {\mathrm{im}(\mathbb{L})} \ar[d]^{\tilde{\beta}|\mathrm{im}(\mathbb{L})} \ar[r]^{ } & 0  \\
  0 \ar[r]^{ } & {\tilde{\Phi}}\cap \prod(\mu\times U^{\mathrm{ab}}) \ar[r]^{ } & {\tilde{\Phi}} \ar[r]^{\prod \mathbb{L}_U} &   {\mathrm{im}(\tilde{\Phi})}\ar[r]^{ } &0   }\]
analogous to \eqref{5termdiagram}. The map $\iota$ is an isomorphism by (RW3a) (cp.\ \cite[lem.\
4.9]{sch-venkakde}) while $\tilde{\beta}$ is an isomorphism by Theorem \ref{t_G} combined with
\eqref{imageL} and Lemma \ref{injectivemodtrace} above.
\end{proof}

Note that in  contrast to Kakde's approach an explicit description of the image of $\tilde{\beta}$
is avoided in the above proof. As Kakde pointed out to me such description would be very messy,
because the description of the image of $\tilde{\theta}$ is - among others - in terms of the
additive congruences (RW3) which would translate into terms of the logarithm applied to them for
the image of $\tilde{\beta}$.

Recently, also the kernel of $\mathbb{L}_{B(\mathcal{G})}$ has been determined by Kakde
\cite{kakde-completion} using the result for abelian $\mathcal{G}$ by Ritter and Weiss in
\cite{RWVI}:

\[\ker(\mathbb{L}_{B(\mathcal{G})})=SK_1(B(\mathcal{G}))\times \mu \times \mathcal{G}^{\mathrm{ab}},\]
where \[SK_1(B(\mathcal{G})):=\ker\left( K_1(B(\mathcal{G}))\to
K_1(B(\mathcal{G})[\frac{1}{p}])\right).\] Hence the same proof as above also shows the following

\begin{theorem}\label{imageB}
\[K'_1(B(\G)):= K_1(B(\G))/SK_1(B(\G))\cong\tilde{\Phi}_B.\]
\end{theorem}

{\bf Question:} Does an analogous statement hold for $A(\mathcal{G})?$

Finally we want to remark that -   fixing a central subgroup $\Z\cong\mathbb{Z}_p$ of $\G$
contained in $\Gamma$ -
 we may replace the infinite index sets $S_\G,S_\G^{\mathrm{ab}}$ by the
finite subsets
\begin{eqnarray*}
 S_{\G,\Z} &:=& \{U| \Z\subseteq U\subseteq \G/C\mbox{
for some } C\trianglelefteq\G \mbox{ with } C\subseteq H\}
\end{eqnarray*} (note that $\Z$ can be considered canonically as subgroup of $\G/C,$ because $\Z\cap C=1$) and
\[ S_{\G,\Z}^{\mathrm{ab}}:=\{U\in S_{\G,\Z}| U\mbox{ abelian}\}.\]

For the corresponding $\widetilde{\Phi}$ we obtain the same statements as in Theorems \ref{image}
and \ref{imageB}.



\def\Dbar{\leavevmode\lower.6ex\hbox to 0pt{\hskip-.23ex \accent"16\hss}D}
  \def\cfac#1{\ifmmode\setbox7\hbox{$\accent"5E#1$}\else
  \setbox7\hbox{\accent"5E#1}\penalty 10000\relax\fi\raise 1\ht7
  \hbox{\lower1.15ex\hbox to 1\wd7{\hss\accent"13\hss}}\penalty 10000
  \hskip-1\wd7\penalty 10000\box7}
  \def\cftil#1{\ifmmode\setbox7\hbox{$\accent"5E#1$}\else
  \setbox7\hbox{\accent"5E#1}\penalty 10000\relax\fi\raise 1\ht7
  \hbox{\lower1.15ex\hbox to 1\wd7{\hss\accent"7E\hss}}\penalty 10000
  \hskip-1\wd7\penalty 10000\box7} \def\Dbar{\leavevmode\lower.6ex\hbox to
  0pt{\hskip-.23ex \accent"16\hss}D}
  \def\cfac#1{\ifmmode\setbox7\hbox{$\accent"5E#1$}\else
  \setbox7\hbox{\accent"5E#1}\penalty 10000\relax\fi\raise 1\ht7
  \hbox{\lower1.15ex\hbox to 1\wd7{\hss\accent"13\hss}}\penalty 10000
  \hskip-1\wd7\penalty 10000\box7}
  \def\cftil#1{\ifmmode\setbox7\hbox{$\accent"5E#1$}\else
  \setbox7\hbox{\accent"5E#1}\penalty 10000\relax\fi\raise 1\ht7
  \hbox{\lower1.15ex\hbox to 1\wd7{\hss\accent"7E\hss}}\penalty 10000
  \hskip-1\wd7\penalty 10000\box7} \def\Dbar{\leavevmode\lower.6ex\hbox to
  0pt{\hskip-.23ex \accent"16\hss}D}
  \def\cfac#1{\ifmmode\setbox7\hbox{$\accent"5E#1$}\else
  \setbox7\hbox{\accent"5E#1}\penalty 10000\relax\fi\raise 1\ht7
  \hbox{\lower1.15ex\hbox to 1\wd7{\hss\accent"13\hss}}\penalty 10000
  \hskip-1\wd7\penalty 10000\box7}
  \def\cftil#1{\ifmmode\setbox7\hbox{$\accent"5E#1$}\else
  \setbox7\hbox{\accent"5E#1}\penalty 10000\relax\fi\raise 1\ht7
  \hbox{\lower1.15ex\hbox to 1\wd7{\hss\accent"7E\hss}}\penalty 10000
  \hskip-1\wd7\penalty 10000\box7}
\providecommand{\bysame}{\leavevmode\hbox to3em{\hrulefill}\thinspace}
\providecommand{\MR}{\relax\ifhmode\unskip\space\fi MR }
\providecommand{\MRhref}[2]{%
  \href{http://www.ams.org/mathscinet-getitem?mr=#1}{#2}
} \providecommand{\href}[2]{#2}

\end{document}